\documentclass{article}

\usepackage{hyperref}
\usepackage{amsmath}
\usepackage{amsthm}
\usepackage{amsfonts}
\usepackage{graphicx}

\usepackage[T1]{fontenc}
\usepackage[utf8]{inputenc}
\usepackage{multirow}

\usepackage{listings}
\usepackage{color}

\newtheorem{thm}{Theorem}

\newtheorem{definition}{Definition}

\title{Analysis of a new type of fractional linear multistep method of order two  with improved stability}

\author{H. M. Nasir$^1$ \and Khadija Al Hasani$^2$}
\date{%
    $^1$FracDiff Research Group, Department of Mathematics, Sultan Qaboos University, Sultanate of Oman  \\%
    $^2$ Department of Mathematics, Sultan Qaboos University, Sultanate of Oman\\[2ex]%
    \today
}

\begin{document}

\maketitle

\begin{abstract}
We present and investigate a new type of implicit fractional linear multistep method of order two for fractional initial value problems. The method is obtained from the
second order super convergence of the Gr\"unwald-Letnikov approximation of the fractional derivative at a non-integer shift point. The proposed method is of order two consistency and coincides with the backward difference method of
order two for classical initial value problems when the order of the derivative is one. The weight coefficients of the proposed method are obtained from the Gr\"unwald weights and hence computationally efficient compared with that of the fractional backward difference formula of order two. The stability properties are analyzed and shown that the stability region of the method is larger than that of the fractional Adams-Moulton method of order two and the fractional trapezoidal method.
Numerical result and illustrations are presented to justify the analytical theories.
\end{abstract}



\renewcommand{\thefootnote}{\fnsymbol{footnote}}

\parindent 0pt
\parskip 5pt

\section{Introduction}

Consider the fractional initial value problem (FIVP)
\begin{subequations}\label{Eq:FInit}
\begin{align}
\;^C_{t_0} D^\beta_t y(t) &= f(t,y(t)), \quad t \ge t_0, \quad 0<\beta \le 1,\label{Eq_FIVP}\\
y(t_0) &=  y_0 ,\label{Eq_FInit}
\end{align}
\end{subequations}

where $\;^C_{t_0} D^\beta_t$ is the left Caputo fractional derivative operator defined in Section \ref{Sec_prem}, $f(t,y)$  is a source function
satisfying Lipschitz condition in the second argument $y$ guaranteeing a unique solution to the
problem \cite{diethelm2010analysis}.

Fractional calculus, despite its long history, have only recently gained places in science, engineering, artificial intelligence and many other fields.

Many numerical methods have been developed in the recent past for solving \eqref{Eq:FInit}
approximately. We are interested in the numerical methods of type
commonly known as the fractional linear multistep methods
(FLMM).

The basic numerical method of FLMM type of order one for \eqref{Eq:FInit} is obtained from the Gr\"unwald-Letnikov form for the fractional derivative \cite{podlubny1998fractional, oldham1974fractional}. The weight coefficients for this basic FLMM are the {\it Gr\"unwald weights} obtained from the series of the generating function $(1-z)^\beta $.

Lubich \cite{lubich1985fractional} introduced a set of higher order FLMMs as convolution quadratures for the Volterra integral equation
(VIE) obtained by reformulating \eqref{Eq:FInit} (See also eg. \cite{diethelm2010analysis}). The quadrature coefficients are obtained from the fractional order power of the rational polynomial obtained from the generating polynomials of linear multistep method (LMM) for classical initial value problems (IVPs).
As a particular subfamily of these FLMMs, the fractional
backward difference formulas (FBDFs) were also proposed by Lubich in \cite{lubich1986discretized}.
Another particular form of FLMM type is the fractional trapezoidal
method of order 2.

Several authors have utilized these formulations to construct
variations of the FLMMs, see eg. \cite{galeone2006multistep} and the references therein.
Galeone and Garrappa \cite{galeone2008fractional} studied
some implicit FLMMs generalizing the Adams-Moulton methods
for classical IVPs.
Galeone and Garrappa \cite{galeone2009explicit} and Garrappa \cite{garrappa2009some} have investigated a set of explicit FLMMs
generalizing the Adams-Bashforth methods.

In \cite{BrazilSymp2022}, the authors constructed a new type of FLMM of order 2 that does not fall under the above mentioned subfamilies of FLMMs and presented an extended abstract in
\cite{Brazil2022}.

In this paper,
we analyse the method for computation and stability, and present an algorithm. We also compare the method with other known FLMMs of order 2 and show that the presented method outweighs the other methods in stability and/or computational efficiency.

This paper is organized as follows. In Section \ref{Sec_prem}, the preliminaries and previous relevant works are summarized.  In Section \ref{Sec_New_FLMM}, the new FLMM
of order 2 is introduced along with a computational algorithm. Numerical examples for testing the method are given in
Section \ref{Sec_Numer_Test}. In Section \ref{Sec_Anal_Stab},
the stability of the method is analysed. In Sections \ref{Sec_Compare}, the new method is compared with other
FLMMs and Section \ref{Sec_Conclusion} draws some conclusions.

\section{Preliminaries}\label{Sec_prem}

The Riemann-Liouville fractional integral of order $\beta>0 $ of a function $f(t)$ in an interval domain $[t_0, T) $ ( $T$ can also be infinity)  is defined as
\begin{equation}\label{Eq_RL_Integral}
  J_t^\beta f(t) =  \frac{1}{\Gamma(\beta)} \int_{t_0}^{t} (t-\tau)^{\beta-1} f(\tau) d\tau,
\end{equation}
where $\Gamma(\cdot) $ denotes the Euler-Gamma function.

For a sufficiently smooth function $y(t) $ defined for  $ t\ge t_0$, the left Riemann-Liouville (RL) fractional derivative of order $\beta>0$ is defined by (see eg. \cite{podlubny1998fractional})
\begin{equation}\label{Eq_R_L}
\;^{RL}_{t_0} D^\beta_t y(t) = \frac{1}{\Gamma(m-\beta)} \frac{d^m}{dx^m}
\int_{t_0}^t \frac{y(\tau)}{(t-\tau)^{\beta - m +1}} d\tau, \quad  m-1 <\beta \le m,
\end{equation}
where $ m = \lceil \beta \rceil $ -- the smallest integer larger than or equal to $\beta $.

The left Caputo fractional derivative of order $\beta > 0$ is defined as
\begin{equation}\label{Eq_Caputo}
\;^{C}_{t_0} D^\beta_t y(t)  = J^{m-\beta}_t y^{(m)}(t) =  \frac{1}{\Gamma(m-\beta)}
\int_{t_0}^t \frac{y^{(m)} (\tau)}{(t-\tau)^{\beta - m +1}} d\tau, \quad m-1 <\beta \le  m,
\end{equation}
where $y^{(m)} $ is the $m$-th  derivative of $y$.

Often, for practical reasons, the integer ceiling $m$ of the fractional order $\beta$ is considered to be one or two. In this paper, we investigate \eqref{Eq:FInit} for the case $ 0 <\beta \le 1$ when $m = 1$.

In addition to the above two definitions, the Gr\"unwald-Letnikov(GL) definition is useful for numerical approximations of fractional derivatives.
\begin{equation}\label{Eq_GL}
  \,^{GL}_{t_0} D^\beta_t y(t) =  \lim_{h \rightarrow 0} \frac{1}{h^\beta} \sum_{k=0}^\infty g_k^{(\beta)} y(t-kh) ,
\end{equation}
where $g_k^{(\beta)} = (-1)^k\frac{ \Gamma(\beta+1)}{\Gamma(\beta - k +1) k! } $  are the  {\it Gr\"unwald weights}  and are the coefficients of the series expansion of the {\it Gr\"unwald generating function}
\begin{equation}\label{Eq_GGen}
  \omega_1(z) = (1-z)^\beta = \sum_{k=0}^\infty g_k^{(\beta)} z^k.
\end{equation}
The coefficients can be successively computed by the recurrence relation
\[
g_0^{(\beta)}  = 1, \qquad
g_k^{(\beta)}  = \left( 1 - \frac{\beta +1}{k}\right)g_{k-1}^{(\beta)} ,
\quad k = 1, 2,...\;.
\]

For theoretical purposes, the function $ y (t) $ has been zero extended for $ t<t_0 $
and hence the infinite summation in the GL definition \eqref{Eq_GL}.
Practically, the upper limit of the sum is
$ n = [(t-t_0)/h]  $, where $[x] $ denotes the integer part of $x$.

The  three definitions in \eqref{Eq_R_L}--\eqref{Eq_GL} are equivalent under homogeneous derivative conditions at the initial point $t_0$ \cite{podlubny1998fractional}.

\subsection{Approximation of fractional integrals and derivatives}
Numerical approximation of the fractional
integral \eqref{Eq_RL_Integral} is commonly
considered via convolution quadrature formulas of the form
\begin{equation}\label{Eq:Conv_Quad}
  J_h^\beta f(t)  = h^\beta \sum_{k=0}^{n} \omega_k f_{n-k},
\end{equation}
where the interval $[t_0,t] $ is discretized by the points set $\{ t_0, t_1, ...,t_n\} $ with
$ t_k = t_0 + kh, f_k = f(t_k),$ for $ k = 0,1,...,n, $ and $   h = (t-t_0)/n $.
The weights $\omega_k $ are from the quadrature rule applied.

For numerical approximation of the fractional derivative, the GL definition is
commonly used by dropping the limit in \eqref{Eq_GL} resulting in the Gr\"unwald  approximation (GA)  for a fixed step $ h $ \cite{oldham1974fractional}.

\begin{equation}\label{Eq_GA}
   \delta_{h}^{\beta}  y(t) := \frac{1}{h^\beta} \sum_{k=0}^\infty g_k^{(\beta)} y(t-k h) .
\end{equation}

It is known that the GA is of order one \cite{meerschaert2004finite, nasir2013second}.
\begin{equation}\label{Eq_GAOrder}
    \delta_{h}^{\beta}  y(t) := \,^{GL}_{t_0} D^\beta_t y(t) + O(h).
\end{equation}

A more general Gr\"unwald  type approximation is given by the shifted Grunwald approximation (SGA) \cite{meerschaert2004finite}.
\begin{equation}\label{Eq_shiftGA}
   \delta_{h, r}^{\beta}  y(t) = \frac{1}{h^\beta} \sum_{k=0}^\infty g_k^{(\beta)} y(t-(k-r) h),
\end{equation}
where $r$ is the shift parameter.

For an integer shift $r$, the SGA is also of first order approximation
\cite{meerschaert2004finite}.
\[
    \delta_{h, r}^{\beta}  y(t) =  \,^{GL}_{t_0}D^\beta_t y(t) + O(h),   \quad r \in \mathbb{Z}.
\]

However, it is observed that the SGA gives a second order approximation
at a non-integer shift $ r = \beta/2 $ displaying super convergence \cite{nasir2013second}.
\begin{equation}\label{Eq_Super}
    \delta_{h, \beta/2}^{\beta}  y(t) =  \,^{GL}_{t_0}D^\beta_t y(t) + O(h^2).
\end{equation}

Some higher order Gr\"unwald type approximations with shifts were presented in \cite{gunarathna2019explicit} with the weight coefficients obtained from some generating functions given in an explicit form according to the order and shift requirements.

\subsection{Fractional initial value problem}
For $ m-1 < \beta \le m $, the general form of a FIVP is given by
\begin{subequations}\label{Eq_FIVP_t_0}
\begin{align}
\;^{C}_{t_0} D^\beta_t \bar{y}(t) &= \bar{f}(t,\bar{y}(t)),\quad t \ge t_0, \quad m-1 < \beta \le m ,\label{Eq_FIVP1}\\
\bar{y}^{(k)}(t_0) &=  y_0^{(k)}, \quad k = 0, 1,..., m-1 \label{Eq_FInit1}.
\end{align}
\end{subequations}

Without loss of generality, one may consider the FIVP with homogeneous initial conditions at the origin ($t_0 = 0$):
\begin{subequations}\label{Eq_FIVP_0}
\begin{eqnarray}
\,^{RL}_{\;\; 0} D^\beta_t y(t) &=& f(t,y(t)), \quad t \ge  0, \quad m-1 < \beta \le m ,\label{Eq_FIVP2}\\
y^{(k)}(0) &=&  0, \quad  k = 0, 1,..., m-1, \label{Eq_FInit2}
\end{eqnarray}
\end{subequations}
because $ y(t)$ is a solution of \eqref{Eq_FIVP_0} if and only if
\begin{equation} \label{Eq_Sol1}
\bar{y}(t) = y(t-t_0) + T_{m-1} (t-t_0)
\end{equation}
is the solution of \eqref{Eq_FIVP_t_0}, where
\[
T_{m-1} (t) = \sum_{k=0}^{m-1} \frac{t^k}{k!}y_0^{(k)} , \quad
f(t,y) = \bar{f}(t+t_0,\bar{y}(t+t_0)).
\]
\color{black}
For $ 0< \beta \le 1 $, we have $ T_0(t) = y_0$ and $ f(t,y) = \bar{f} (t+t_0, y(t+t_0) +y_0) $.

The problem \eqref{Eq_FIVP_0}  can be equivalently expressed by the VIE of second kind \cite{diethelm2010analysis}
\begin{equation}\label{Eq_VIE}
  y(t) = \frac{1}{\Gamma(\beta)} \int_0^t
  (t-\eta)^{\beta - 1} f(\eta , y(\eta))  d\eta.
\end{equation}

\subsection{Fractional linear multistep methods}

Among the several numerical methods to solve \eqref{Eq_FIVP_0} and thus \eqref{Eq_VIE}, we list the numerical methods that fall under the category of FLMM.

{\bf The Gr\"unwald-Letnikov method:}
The fundamental and widely investigated numerical approximation scheme for the FIVP \eqref{Eq_FIVP_0} is the Gr\"unwald-Letnikov method (also called
{\it fractional backward Euler method})
 obtained by replacing the fractional derivative operator in (\ref{Eq_FIVP2}) by its GA
operator $\delta^\beta_h $ in
(\ref{Eq_GA}) with (\ref{Eq_GAOrder}) \cite{podlubny1998fractional}.
\begin{equation}\label{Eq_FEuler}
  \frac{1}{h^\beta} \sum_{k=0}^\infty g_k^{(\beta)} y(t-k h)
  = f(t,y)  + O(h).
\end{equation}
By choosing the discretization step $h $ appropriately to align the discrete points $ t- kh$  with the end points of the problem domain $[0,T] $ and assuming zero extension for the unknown function $y(t) $ for $t < 0$, the infinite sum in \eqref{Eq_FEuler} is reduced to a finite sum.
Dropping the first order error term, choosing $h = T/N, N\in \mathbb{N} $ and denoting
\begin{equation}\label{Eq_Notations}
 t_n = nh, \quad  y_n  \approx y(t_n) \quad \text{ and } \quad
f_n = f(t_n, y_n),
\end{equation}
equation \eqref{Eq_FEuler} gives the GL scheme
\begin{equation}\label{Eq_FEuler1}
   \sum_{k=0}^n g_k^{(\beta)} y_{n-k}
  = h^\beta f_n,  \quad  n = 1,2, ...,N.
\end{equation}

The FIVP can also be approximated via its VIE form \eqref{Eq_VIE} by simply replacing the
integral by its approximation \eqref{Eq:Conv_Quad}. Some approximations in this line are the product integration methods \cite{young1954application, cameron1984product, lubich1986stability}.

Lubich  \cite{lubich1986discretized} presented and studied numerical approximation methods  for the VIE \eqref{Eq_VIE} in the form
\begin{equation}\label{Eq_FLMM_Int}
  y_n = h^\beta \sum_{k=0}^{n} \omega_k f_{n-k}, \quad n = 1,2,...
\end{equation}
with weights $\omega_k$ as the coefficients of the series expansion of the generating function
\[
\omega(\xi) = \left(
\frac{\sigma(1/\xi)}{\rho(1/\xi)}
\right)^\beta,
\]
where $(\rho,\sigma)$ is a pair of generating polynomials of a LMM of a prescribed order for classical IVP. However, Lubich observed and showed that for approximations of order more than one, the intended order $p$ is achieved only for a certain class of functions, specifically
for functions of the form $ y(t) = t^{\alpha -1 } g(t) ,  \alpha \ge p $, where $ g(t)$ is analytic.  However, for $\alpha < p,$ the order is reduced to
$O(h^\alpha)$ only. To remedy this order reduction, an additional sum
is introduced in \eqref{Eq:Conv_Quad} to have the approximation scheme
\begin{equation}\label{Eq:VIE_approx}
y_n = {h^\beta} \sum_{k=0}^s w_{n,k} f_{k} +
{h^\beta} \sum_{k=0}^n w_{k} f_{n-k}.
\end{equation}
Here, the {\it starting weights}  $w_{n,k} $  are   to
compensate the reduced order of convergence.

Another way to approximate the FIVP \eqref{Eq_FIVP_0} is to replace
the fractional derivative by its approximation in the form
\eqref{Eq_GA} with general weights $w_k$ as
\begin{equation}\label{Eq:GA_type}
  \Omega_{h}^{\beta}  y(t) := \frac{1}{h^\beta} \sum_{k=0}^\infty w_k y(t-k h),
\end{equation}
where the weights $w_k$ are chosen for a desired order of approximation.
Thus, Grunwald type approximation schemes for the FIVP have the form
expressed in conformance with the classical LMM form as
\begin{equation}\label{Eq_FLMM}
  \sum_{k=0}^n w_k y_{n-k}
   = h^\beta f_n, \quad n = 1,2,... \;.
\end{equation}
{\bf Remark 1:}
Note that, analogous to the case of FLMM for VIE, the FLMM for FIVP also displays the order reduction for the class of functions mentioned for the FLMM for VIE. Therefore, an adjusting sum with some starting weights is added to remedy this situation. We also point out, however, that this additional sum does not affect the convergence and stability of the underlying FLMM. Besides, including this sum in the computation of solution, though it rectifies the order, brings additional difficulties in the implementation such as (i) the number of starting weights vary depending on the fractional order $\beta$,  (ii) computing the starting weights at every iterations, (iii) the system to solve for the starting weights is highly ill conditions, etc.

It can be shown that the approximation schemes \eqref{Eq_FLMM_Int}  and \eqref{Eq_FLMM} are equivalent (see \cite{galeone2008fractional} and the references therein) and
the generating functions $\omega(\xi)$ and $\delta(\xi)$ of the weight coefficients $\omega_k$ and $w_k$ in \eqref{Eq_FLMM_Int}  and \eqref{Eq_FLMM} respectively can be shown to have the relation  $\delta(\xi) = (\omega(\xi))^{-1} $.
Thus, the weights $w_k$ in \eqref{Eq_FLMM_Int} can be chosen as the coefficients of the
series expansion of the generating function
\begin{equation}\label{Eq_Gen_FLMM}
\delta(\xi) = \left(
\frac{\rho(1/\xi)}{\sigma(1/\xi)}
\right)^\beta .
\end{equation}

Lubich \cite{lubich1986discretized} also presented some subclasses of FLMMs for VIE with generating functions of general form
$(r_1(\xi)^\beta r_2(\xi)$, where $ r_1(\xi) , r_2 (\xi) $ are
rational polynomials. Analogously, the FLMMs for FIVP can also be
considered with the generating functions of the form
\begin{equation}\label{Eq:Gen_FLMM_gen}
    \delta(\xi) = \left(\frac{a(\xi)}{b(\xi)}\right)^\beta
    \frac{p(\xi)}{q(\xi)} ,
\end{equation}
where $a, b, p $ and $q $ are polynomials.

\subsection{Stability regions for the FLMM}

The following definitions are fundamental for the analysis of stability of a FLMM.

\begin{definition}\rm [Stability]
Let $y_n$ be a solution of a recurrence relation with
initial data vector ${\bf y}_0$.
\begin{enumerate}
\item
$y_n$ is {\it stable} if for any perturbation $\delta {\bf y}_0 $ in ${\bf y}_0$, the resulting changes $\delta y_n$ in
$y_n$ are uniformly bounded for all $n \in \mathbb{N}$.
\item  The solution is {\it asymptotically stable} if, moreover,  $\delta y_n \rightarrow 0$ as $n \rightarrow \infty$.
\end{enumerate}
\end{definition}

The stability region for FLMM is given by
$$ S_\beta = \{ \lambda h^\beta \in \mathbb{C}\setminus \{0\} : y_n \rightarrow 0 \text{ as } n \rightarrow \infty  \}, $$ where $\lambda $ is a complex parameter of the stability test problem  \ $ {}^C D^\beta y(t) = \lambda y(t), y(0) = y_0 $.

The generating function for an FLMM directly gives
the stability region for the method.
\begin{thm} \cite{lubich1986discretized}
The stability region of an FLMM with generating function $\delta(\xi) $ is given by
\begin{equation}\label{Eq_Stab}
  S = \{ \delta(\xi) : |\xi| >1 \},
\end{equation}
\end{thm}

We list the subfamilies of the FLMMs found in the literature.

\begin{enumerate}
\item {\bf Fractional trapezoidal method:}
The fractional trapezoidal method of order 2 (FT2) obtained from the trapezoidal rule for the ODE has the generating function\cite{lubich1986discretized}
\[
\delta_{FT2}(\xi) = \left(2\frac{1-\xi}{1+\xi}\right)^\beta.
\]
It is the only method known so far in the form
 $\delta(\xi) = \left(\frac{a(\xi)}{b(\xi)}\right)^\beta $
with $ b(\xi) \ne 1, p(\xi) = q(\xi) =1 $ in \eqref{Eq:Gen_FLMM_gen}.

\item {\bf Fractional backward difference formula:}
The fractional backward difference formula (FBDF)  obtained from the BDF for classical IVP has the generating functions of the form
$\delta(\xi) = (a(\xi))^\beta $.

For orders $ 1 \le m\le 6$, a set of 6 FDBF$m$ methods have been obtained with  polynomials corresponding to the generating polynomials of the BDF of order $m$  given by $a(\xi) = \sum_{k=1}^m \frac{1}{k}(1-\xi)^k $.
\item {\bf Fractional Adams methods:}
The fractional Adams methods have the generating functions of the form $\delta(\xi) = \frac{(a(\xi)^\beta}{q(\xi)}  $, where
the polynomial $a(\xi)$ is one of the polynomials in FBDF methods and $q(\xi)$ is determined to have a specified order of consistency for the method. Often, $a(\xi) = 1-\xi $ \cite{galeone2006multistep, galeone2008fractional,
galeone2009explicit,garrappa2009some}.
However, other polynomials in the FBDF have also appeared in the literature \cite{bonab2020higher,heris2018fractional}.

When $q_0 = 0$, the method is explicit and is called  fractional Adams-Bashforth methods (FABs) \cite{galeone2009explicit,garrappa2009some}. $q_0 \ne 0 $ gives implicit methods and are called fractional Adams-Moulton methods (FAMs).
\item {\bf Rational approximation:} In \cite{aceto2015construction}, a classical LMM type of approximation is proposed to obtain a class of FLMMs by rational approximations of the FBDF generating functions in the form $\delta(\xi) = \frac{p(\xi)}{q(\xi)}$.

\end{enumerate}

The order of consistency of a FLMM can also be determined from its
generating function.
\begin{thm}\cite{NasirNafa,NasirNafaANZIAM, gunarathna2019explicit}.
The order of an FLMM
 with generating function $\delta(\xi) $  is $p$ if and only if
\begin{equation}\label{Eq_Order_Gen}
  \frac{1}{x^\beta}\delta(e^{-x}) = 1 + O(x^p).
\end{equation}
Moreover, the approximation corresponding to $\delta(\xi)$ satisfies, with $D_t^\beta $ denoting the RL fractional derivative,
\[
\delta_h^\beta y(t) = D_t^\beta y(t) + h^p a_p(\beta)
 D_t^{\beta+p}  y(t) + h^{p+1} a_{p+1}(\beta) D_t^{\beta+p+1}  y(t) + ...,
\]
where $y(t)$ is assumed to be sufficiently smooth.
\end{thm}

\section{A new fractional linear multistep method}\label{Sec_New_FLMM}

We present the main result of constructing a new FLMM
of order 2.

The fractional derivative of the FIVP   \eqref{Eq_FIVP_0} is replaced by the approximation (\ref{Eq_Super}) with super convergence of order 2. This gives at $t = t_n$,
\begin{equation}\label{Eq_super1}
  \delta_{h, \beta/2}^{\beta}  y(t_n)
  = \frac{1}{h^\beta} \sum_{k=0}^{\infty} g_k^{(\beta)} y(t_n - (k-\beta/2)h)  = f(t_n, y(t_n)) + O(h^2).
\end{equation}

Since $k-\beta/2 $ is not integer for $0< \beta \le 1$, the point $t_n-(k-\beta/2)h =: t_{n-k+\beta/2} $ is not aligned with the discrete points of the computational domain $\{t_m , m = 0, 1,...,N\}$. Replace it with an order 2 approximation with points $t_{n-k}$ and $t_{n-k-1}$ in the computational domain given by
\begin{equation}\label{Eq_Align}
    y\left(t_{n-k+\beta/2}\right) = \left(1+ \beta/2 \right) y(t_n - kh) - (\beta/2) y(t_n - (k-1)h) +O(h^2).
\end{equation}

With the notations in \eqref{Eq_Notations},
we obtain the new implicit FLMM approximation scheme
\begin{equation}\label{Eq_NewFLMM2}
  \sum_{k=0}^{\infty} g_k^{(\beta)}
  \left[ \left(1+ \frac{\beta}{2}\right) y_{n-k} - \frac{\beta}{2} y_{n-k-1} \right]   = h^\beta f_n,  \quad  n = 1,2, \cdots.
\end{equation}

The coefficients in the new FLMM \eqref{Eq_NewFLMM2} are
linear expressions of the Gr\"unwald weights $g_k^{(\beta)}$ and thus does not involve any extra computations.

For the order of the method, we have the following:
\begin{thm}
  The new FLMM in \eqref{Eq_NewFLMM2} is consistent with   order 2.
\end{thm}
\begin{proof}
Immediately follows from \eqref{Eq_Super}, \eqref{Eq_super1} and \eqref{Eq_Align}.
\end{proof}

\begin{thm}
The generating function of the new implicit FLMM is
given by
\begin{equation}\label{Eq_ChFLMM}
  \delta(\xi) = (1-\xi)^\beta p(\xi),
\end{equation}
where
$
p(\xi) = \left( 1+ \frac{\beta}{2} \right)  -
\frac{\beta}{2} \xi.
$
Moreover, the generating function satisfies
\[
\frac{1}{x^\beta} \delta(e^{-x}) = 1 + O(x^2)
\]
confirming order 2 consistency.
\end{thm}

\begin{proof}
The sum on the left side of \eqref{Eq_NewFLMM2} is
manipulated with $p_0 = 1 + \beta/2, \quad p_1 = -\beta/2 $ as follows:
\begin{align}
  \sum_{k=0}^{\infty} g_k^{(\beta)}
   \left( p_0 y_{n-k} + p_1 y_{n-k-1} \right)
   & =  p_0\sum_{k=0}^{\infty} g_k^{(\beta)}
    y_{n-k} +
  p_1\sum_{k=0}^{\infty} g_k^{(\beta)}
    y_{n-k-1}\nonumber\\
   & = p_0 \sum_{k=0}^{\infty} g_k^{(\beta)}
    y_{n-k} +
  p_1 \sum_{k=1}^{\infty} g_{k-1}^{(\beta)}
    y_{n-k}
   =  \sum_{k=0}^{\infty}
   \left( p_0 g_k^{(\beta)}
   + p_1  g_{k-1}^{(\beta)} \right)
   y_{n-k}, \label{Eq:FLMM_Reform}
\end{align}
where we have set $ g_{-1}^{(\beta)} = 0 $.
The weights
\begin{equation}\label{Eq:FLMM_weights}
 w_k = p_0 g_k^{(\beta)} + p_1 g_{k-1}^{(\beta)}, \quad k = 0,1,...\;.
\end{equation}
are the coefficients of the generating function
   \[
   \delta(\xi) = p_0 (1-\xi)^\beta + p_1 \xi (1-\xi)^\beta = (1-\xi)^\beta (   p_0 + p_1 \xi  ).
   \]
Moreover, we have
\[
\frac{1}{x^\beta} \delta(e^{-x}) = 1 - \frac{\beta(3\beta+5)}{24}x^2 + O(x^3).
\]
which completes the proof.
\end{proof}

{\bf Remark 2:} When $\alpha = 1$,
the new FLMM coincides with the BDF2 method $(\rho, \sigma)$ of order 2 for the classical IVP with generating polynomials
$ \rho(\xi) = \frac{3}{2} - 2 \xi +\frac{1}{2} \xi^2 $ and $\sigma(\xi) = 1$.

{\bf Remark 3:} The notion of super convergence and nodal alignment have been applied for space fractional diffusion equations in \cite{nasir2013second} and \cite{zhao2015series}. To the knowledge of the authors, super convergence of Gr\"unwald approximation for time fractional differential equations has not appeared before in the literature.

\subsection{Implementation}
Here, we give two algorithms to compute the approximate solutions for the FIVP for linear and non-linear cases using the new FLMM .

As the starting weights do not affect the convergence and stability, we exclude the starting sum in the algorithms (see also Remark 1). For details of implementing the starting sum, the reader is directed to \cite{garrappa2015trapezoidal}.

For brevity of notations, the convolution of two vectors ${\bf a,b} $ of size $ n+1 $ is denoted by
$ {\bf a*b} = \sum_{k=0}^{n}a_k b_{n-k} $. For a sequence ${\bf a}$, the vector slice $[a(i), a(i+1) , \ldots, a(j) ] $ is denoted by $ {\bf a}_{i,j} $.

We reformulate the new FLMM scheme \eqref{Eq_NewFLMM2} with \eqref{Eq:FLMM_Reform}
and \eqref{Eq:FLMM_weights} as
\begin{equation}\label{Eq:FLMM_Reform2}
  \sum_{k=0}^{n}w_k y_{n-k} = {\bf w}_{0,n} * {\bf y}_{0,n}
  = w_0 y_n + {\bf w}_{1,n} * {\bf y}_{0,n-1} = \lambda h^\beta f_n.
\end{equation}

In the case of linear FIVP, we have $ f(t,y) = \lambda y(t) + s(t) $ for some constant $\lambda $ and function $s(t)$.

We write the scheme \eqref{Eq:FLMM_Reform2} for this case, with $s_n = s(t_n)$,  as
\begin{align*}
w_0 y_n + {\bf w}_{1,n} * {\bf y}_{0,n-1}  = h^\beta (\lambda y_n + s_n)
\Rightarrow
y_n = \frac{1}{w_0 - \lambda h^\beta}
\left[
h^\beta s_n - {\bf w}_{1,n}*{\bf \bar y}_{0,n-1}
\right], \quad n = 1,2,...\; .
\end{align*}
Hence, the algorithm for the linear FIVP is devised as

{\bf Algorithm 1} [For linear FIVP]
\begin{enumerate}
  \item Input $ \lambda , s(t) $ , $ y_0 $ and $ t_0, h, N $.
  \item Compute sequence  ${\bf w }_{0,N} $.
  \item For $ n = 1, 2, \ldots, N $,
   Compute $ y_n = \frac{1}{w_0 - \lambda h^\beta}
\left[ h^\beta s_n - {\bf w}_{1,n}*{\bf \bar y}_{0,n-1} \right] $.
\end{enumerate}

For non-linear FIVP, the non-linear equation  \eqref{Eq:FLMM_Reform2} in $y_n$ needs to be solved for the unknown $y_n$. The Newton-Raphson method is used to numerically solve this with an initial seed $y_{n, 0}= y_{n-1} $. Thus, the following algorithm results for non-linear
FIVP.

{\bf Algorithm 2} [For non-linear FIVP]
\begin{enumerate}
  \item Input $ f(t,y), f_y(t,y) $ , $ y_0 $
  \item For $ n = 1, 2, \ldots, N $,
  \item $ c_n = {\bf w}_{1,n}*{\bf \bar y}_{0,n-1} $
  \item Set $ y_{n,0} = y_{n-1} $.
  \item For $ k = 1,2,\ldots , $
  \item Compute $ F_{k-1} =
          w_0 y_{n,k-1} -    h^\beta f(t_{n}, y_{n,k-1} ) + c_n $.
  \item Compute $ JF_{k-1} = w_0 -   h^\beta f_y(t_{n}, y_{n,k-1} )$
  \item Compute $ y_{n,k} = y_{n,k-1} - \frac{F_{k-1}}{JF_{k-1}} . $
  \item Until convergence at $k = K $.
  \item Set $ y_n = y_{n,K} $.
\end{enumerate}

\section{Numerical Tests }\label{Sec_Numer_Test}

We used the   new FLMM to compute approximate solutions of the non-linear FIVP
\begin{align*}
D^\beta y(t) &= f(t,y) , \quad 0\le t\le 1, 0<\beta\le 1,\\
y(0) &= 0.
\end{align*}
where
\[
f(t,y) = \frac{\Gamma(2\beta+5)}{\Gamma(\beta+5)}t^{\beta+4}
- \frac{240}{\Gamma(6-\beta)} t^{(5-\beta)}
+ (t^{2\beta+4}-2t^5)^2- y(t)^2.
\]

The exact solution of the problems is given by
$ y(t) = t^{2\beta+4}-2t^5 $.

The problem is solved with fractional orders $\beta = 0.4,0.6, 0.8 $ and $1.0 $. The computational domain of the problem is  $ \{ t_n = n/M,  n = 0,1,\cdots, M\} $ and
 step size $ h = 1/M$, where $M$ is the number of subintervals of the problem domain $[0,1]$.  The problem was computed for $ M = 2^j, j = 3,4,...,12 $.

 The computational order of the method is computed by the formula
 \[
 p_{j+1} = \log( E_{j+1}/E_j )/\log( h_{j+1}/h_j )
 \]
 where $E_j , h_j $ are the Maximum error and the step size for $M = 2^j$.

 Table \ref{Tbl_Orders} list the results obtain in the computations.

 \begin{table}[h]\centering
\footnotesize
 \begin{tabular}{|r||cc||cc||cc|cc|}
 \hline
 & $\beta = 0.4 $ & & $\beta = 0.6$ & & $\beta = 0.8$ & & $\beta = 1.0 $ &\\
 \hline
 $M$ &  Max. Error  & Order & Max Error & Order & Max Error & Order & Max Error & Order \\
 \hline
   8  &  1.698e-01 & -- & 9.070e-02 & -- & 7.835e-02 & -- & 6.985e-02 & -- \\
   16  &  2.779e-02 & 2.61128 & 2.169e-02 & 2.06382 &  1.978e-02 & 1.98599  &  1.769e-02 & 1.98155 \\
   32  &  6.648e-03 & 2.06349 & 5.503e-03 & 1.97912  &  5.060e-03 & 1.96667 &  4.466e-03 & 1.98563\\
   64  &  1.663e-03 & 1.99866 & 1.398e-03 & 1.97644 &  1.286e-03 & 1.97645 &  1.122e-03 & 1.99286 \\
  128  &  4.186e-04 & 1.99047 & 3.534e-04 & 1.98446 &  3.245e-04 & 1.98628 &  2.812e-04 & 1.99660\\
  256  &  1.052e-04 & 1.99271 & 8.888e-05 & 1.99117 &  8.155e-05 & 1.99260 &  7.037e-05 & 1.99836\\
  512  &  2.638e-05 & 1.99566 & 2.229e-05 & 1.99530 &  2.044e-05 & 1.99616 &  1.760e-05 & 1.99920\\
 1024  &  6.605e-06 & 1.99764 & 5.583e-06 & 1.99758 &  5.117e-06 & 1.99804 &  4.402e-06 & 1.99960\\
 2048  &  1.653e-06 & 1.99877 & 1.397e-06 & 1.99877 &  1.280e-06 & 1.99901 &  1.101e-06 & 1.99980 \\
 4096  &  4.133e-07 & 1.99938 & 3.494e-07 & 1.99938 &  3.202e-07 & 1.99950 &  2.752e-07 & 1.99990\\
 \hline
 \end{tabular}
 \caption{Computational order of the new FLMM}\label{Tbl_Orders}
 \end{table}

\section{Analysis of linear stability}\label{Sec_Anal_Stab}

For the analysis of stability of a FLMM, we have the following preparations.
The analytical solution of the test problem
\[
{}^C D_t^\beta y(t) = \lambda y(t), \quad y(0) = y_0
\]
   is given by
$y(t) = E_\beta( \lambda t^\beta)y_0 $, where $E_\beta (\cdot) $ is the
the Mittag-Leffler function
$   E_\beta(x) = \sum_{k=0}^{\infty} \frac{x^k}{\Gamma(\beta k +1)}.
$

The analytical solution $y(t)$ of the test problem is stable in the sense that it vanishes in the $\beta \pi$-angled region
\[
    \Sigma_\beta = \left\{ \xi \in \mathbb{C} :
    |\arg(\xi)| > \frac{\beta\pi}{2}  \right\}.
\]
The analytical unstable region is thus
the infinite wedge $\{ \xi \in \mathbb{C} :
    |\arg(\xi)| \le \frac{\beta\pi}{2}  \} = \mathbb{C}\setminus \Sigma_\beta $.

For the numerical stability of FLMM, we have the following criteria:

\begin{definition}
Let $ S $ be the numerical stability region of a FLMM.
For an angle $\alpha$, define the wedge
\[
 S(\alpha) = \{ \xi :  |\arg(\xi) -\pi| \le \alpha  \}
 = \{ \xi :  |\arg(\xi)| > \alpha  \}.
\]
The FLMM is said to be
\begin{enumerate}
\item
$A(\alpha)$-stable if  $ S(\alpha) \subseteq S $.
\item
$A$-stable if it is $A(\beta \pi/2) $-stable. That is, $ \Sigma_\beta \subseteq S $.
\item
unconditionally stable if the negative real line $(-\infty, 0) \subseteq S$.

 \end{enumerate}
\end{definition}

We analyse the stability of the new FLMM
through its stability region
\[
    S = \{ (1-\xi)^\beta p(\xi) : |\xi| > 1\}
    = \mathbb{C} \setminus S^c ,
\]
where $ S^c = \{ (1-\xi)^\beta p(\xi) : |\xi| \le 1 \} $ is the unstable region.

\begin{thm}\label{Th_bdd_Sym_right}
The unstable region $S^c$ is bounded and symmetric about the real axis.
Moreover, For $ 0 < \beta \le 1$, if $\Im(\xi) > 0$, then
$\Re \delta(\xi) = \Re \delta(\bar \xi) > 0 $
and  $\Im \delta(\xi) = - \Im \delta(\bar \xi ) < 0 $.
\end{thm}

\begin{proof}
For the boundedness, we see that, for $ |\xi| \le 1$,
$
|\delta(\xi)| \le (1 + |\xi|)^\beta \left[ \left(1+\frac{\beta}{2}\right) + \frac{\beta}{2} |\xi| \right] \le 2^\beta (1+\beta) <\infty
$.

For the symmetry about the real axis, we immediately see that
$ \delta(\bar{\xi}) = \overline{\delta(\xi)}$.

Again, for $\xi = e^{i \theta} $, we have
$
 1- \xi = \left( e^{-\frac{i\theta}{2}} - e^{ \frac{i\theta}{2}}   \right)e^{ \frac{i\theta}{2}}
= 2i \sin\frac{\theta}{2} e^{\frac{i\theta}{2}}
= 2 \sin\frac{\theta}{2} e^{i( \frac{\theta}{2}-
\frac{\pi}{2}) }
  =: b e^{i\phi},
$\\
where $\phi \equiv \phi(\theta) = \frac{\theta}{2}-\frac{\pi}{2}  $ and
$  b = 2 \sin\frac{\theta}{2} = 2\cos\phi > 0 $ for $0 < \theta < \pi $.

Since
$
\delta(\xi) = (1-\xi)^\beta [1+ \frac{\beta}{2}(1-\xi)]
$,
we have for its real part,
\begin{equation}\label{Eq:Real_delta}
\Re{(\delta(\xi))} = b^\beta
[\cos \beta \phi + \beta \cos\phi \cos (\beta+1)\phi]
           =: b^\beta g(\theta).
\end{equation}
where, with some trigonometric manipulations,
\[
g(\theta) = \left(1+\frac{\beta}{2}\right)\cos\beta \phi
        + \frac{\beta}{2}\cos(\beta +2) \phi.
\]
Now,
$
g'(\theta) = -\beta\left(1+\frac{\beta}{2}\right)
[ \sin \beta \phi + \sin (\beta +2)\phi ] \phi'
 = -\beta(2+\beta)
\sin(\beta+1)\phi \cos \phi
 > 0,
$
because for $0<\theta < \pi $,
$ \phi \in (-\pi/2 , 0) $
where $\cos\phi > 0$ and $ (\beta+1)\phi \in
[-(\beta+1)\pi/2 , 0]  $ in the quadrants III and IV for $0<\beta \le 1$
where $\sin (\beta+1)\phi < 0 $.
Hence, $g(\theta)$ is increasing with $g(0) = \cos (\beta \pi/2) >0 $.
Thus, $g(\theta) > 0 $ for $ 0 < \theta < \pi $.
It then follows from the symmetry that
$\Re\delta(\xi) = \Re\delta(\bar\xi) > 0 $ .

For the imaginary part of $\delta(\xi) $,
\begin{equation}\label{Eq:Imag_delta}
\Im \delta(\xi) = b^\beta
[\sin \beta \phi + \beta \cos\phi \sin (\beta+1)\phi]
          =: b^\beta h(\theta) < 0,
\end{equation}
because, when $ 0 < \theta < \pi$, we see $\phi$ and $ \beta\phi  $ are in the quadrant IV where
$ \cos\phi >0 $ and
$ \sin \beta \phi < 0 $,
 and $ (\beta+1)\phi$ is in the quadrants III and IV where $ \sin (\beta+1)  \phi < 0 $.
This gives, along with the symmetry about real axis, that
$
  \Im(\delta(\xi(\theta)) = - \Im(\delta(\xi(-\theta)) )<  0 \quad   \text{ for }  0 < \theta < \pi
$
 and the proof is completed.
\end{proof}

Theorem \ref{Th_bdd_Sym_right}(1) tells us that the new FLMM  is $A(\frac{\pi}{2})$-stable.
In fact, we have a stronger result.

\begin{thm}\label{Th_A_Stable}
The FLMM in \eqref{Eq_NewFLMM2} is $A$-stable for $0<\beta \le1 $.
\end{thm}

\begin{proof}
From \eqref{Eq:Real_delta} and \eqref{Eq:Imag_delta},
the tangent at $\theta \in [0,\pi]$ on the stability boundary
$\{ \delta(\xi) : |\xi| = 1 \} $ is $h(\theta)/g(\theta)$ with the derivative
\[
\frac{d}{d\theta} \frac{h(\theta)}{g(\theta)} = \frac{\beta(\beta+1)(\beta+2)\cos^2\phi}{2(g(\theta))^2} > 0.
\]
Thus, the tangent is monotonically increasing in $[0,\pi]$ with the minimum $(h/g)(0) =  -\tan(\beta\pi/2)$  at $\theta = 0$.

Therefore, from the symmetry, the unstable region is contained in the wedge $ \{ \xi : |\arg(\xi)| \le \frac{\beta\pi}{2} \}= \mathbb{C}\setminus \Sigma_\beta$
meaning that the new FLMM is $A$-stable.
\end{proof}

The $A$-stability indicates that, for $0 <\beta \le 1$,
our new FLMM is $A(\pi/2)$-stable and hence unconditionally stable.
\begin{figure}[h]
  \centering
  \includegraphics[width=13cm]{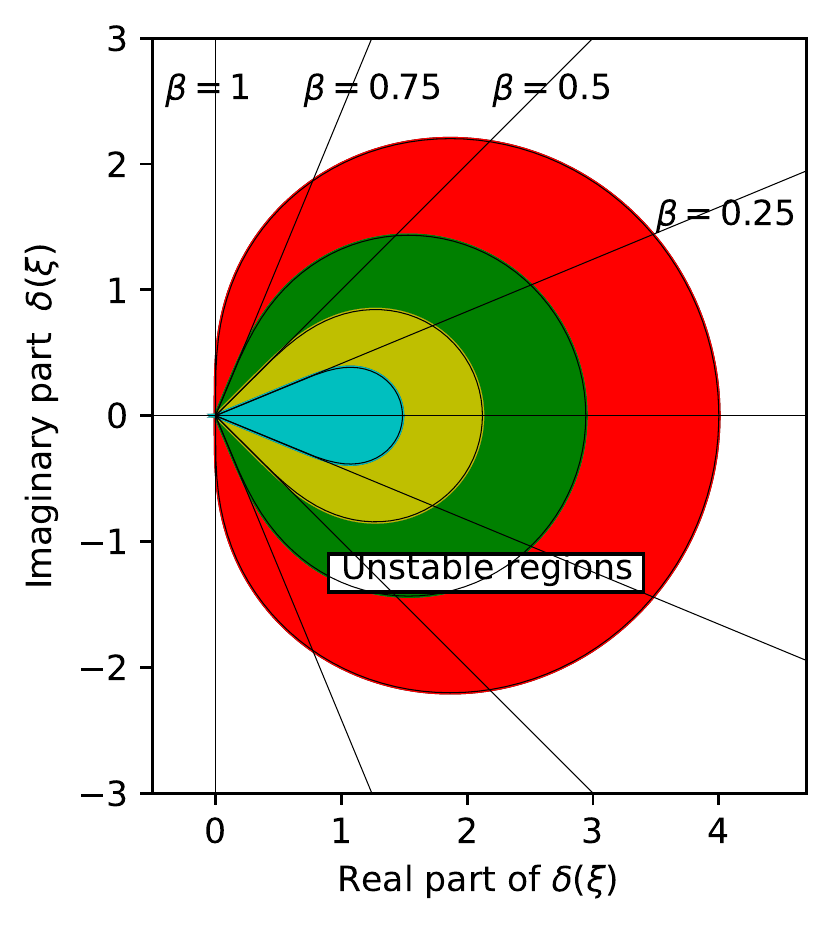}
  \caption{Unstable regions and $A$-stable tangent boundaries for the new FLMM}\label{Fig_A_Stab}
\end{figure}
In Figure \ref{Fig_A_Stab}, the unstable regions and the A-stable tangent boundaries for fractional order values $\beta = 0.25, 0.5, 0.75, 1 $ are shown.

\section{Comparison of stability regions}\label{Sec_Compare}

We compare the stability regions of previously established implicit FLMMs of order 2 with our new FLMM which we now denote by NFLMM2  for want of an abbreviation.

For this, we consider the Lubich's fractional backward difference method FBDF2 \cite{lubich1986discretized}, the fractional Adams-Moulton method FAM1 \cite{galeone2008fractional} and the fractional Trapezoidal rule (FT2) \cite{lubich1986discretized},
\cite{garrappa2015trapezoidal} given by their respective generating functions
\[
\delta_{FBDF2}(\xi) = \left(\frac{3}{2} - 2\xi + \frac{1}{2}\right)^\beta,\quad
\delta_{FAM1}(\xi) = \frac{(1-\xi)^\beta}
{ (1 - \frac{\beta}{2}) + \frac{\beta}{2} \xi  }
\quad \text{ and }\quad
\delta_{FT2}(\xi) = \left( 2\frac{1-\xi}{1+\xi} \right)^\beta.
\]
\begin{figure}[h]
  \centering
  \begin{tabular}{cc}
  \includegraphics[width=6cm]{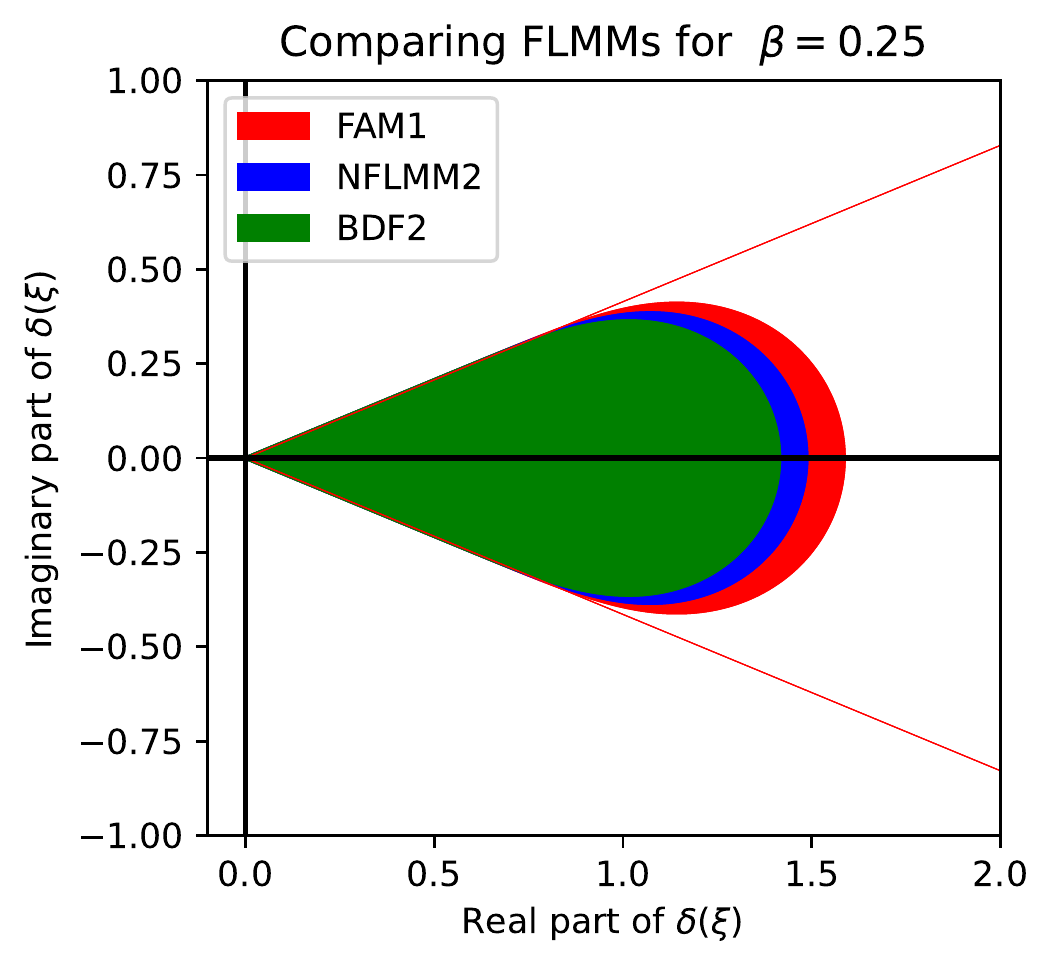}  &
  \includegraphics[width=6cm]{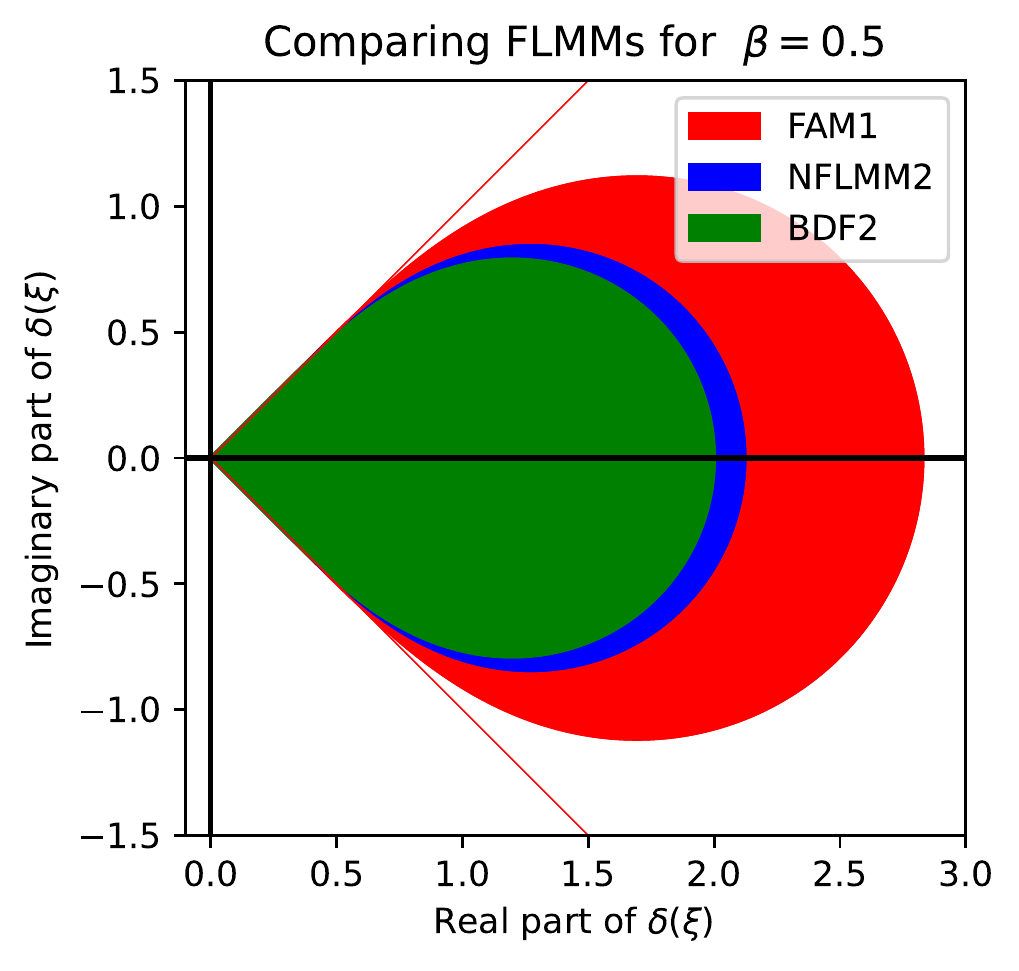} \\
  (a) $\beta = 0.25$ &(b)  $\beta = 0.50$\\
  \includegraphics[width=6cm]{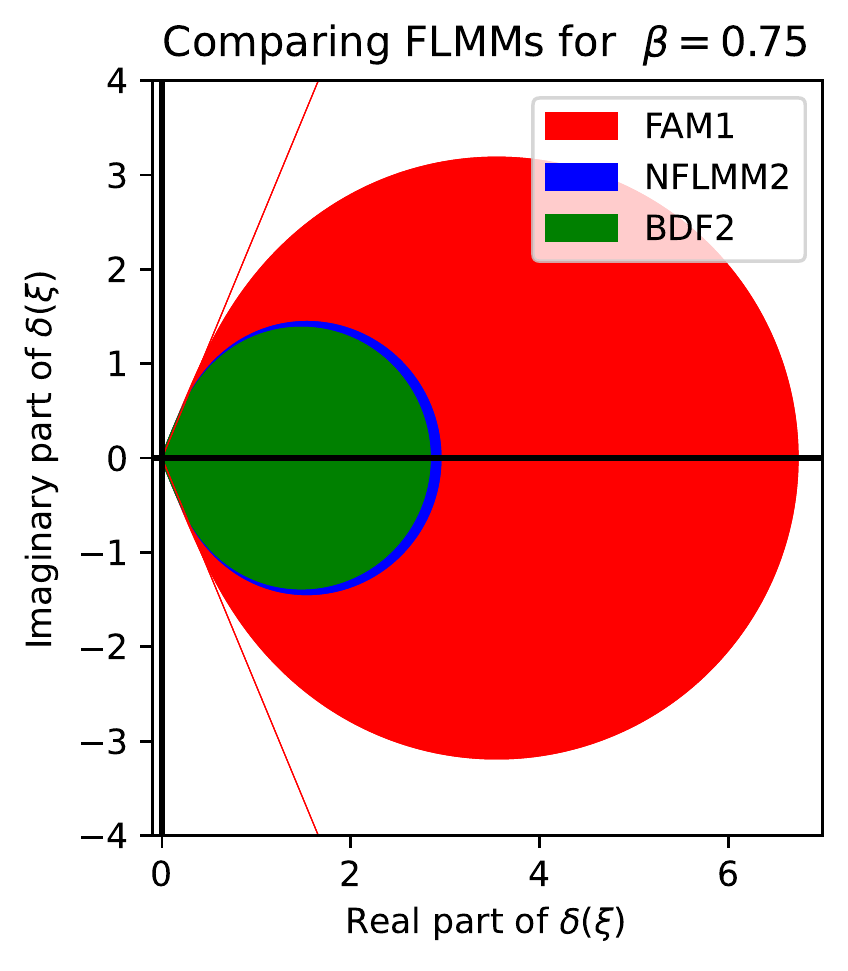} &
  \includegraphics[width=6cm, height = 6.6cm]{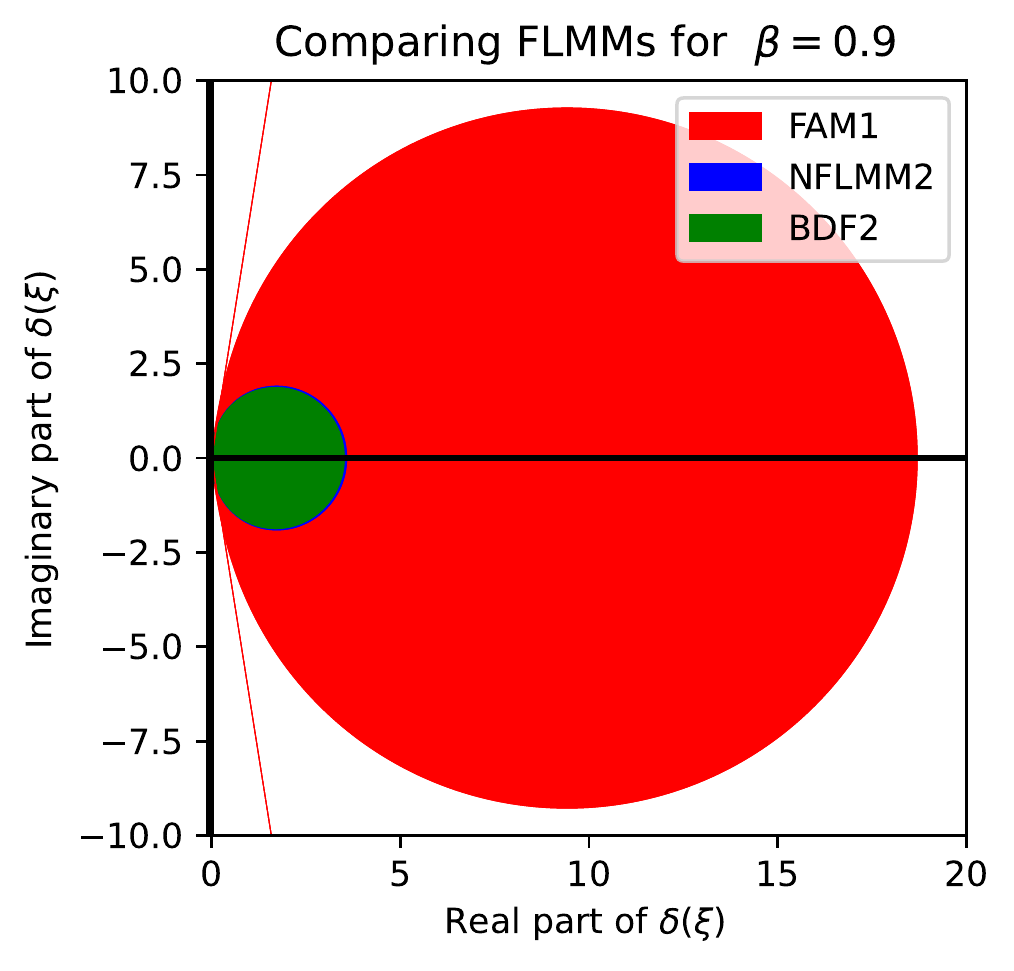} \\
(c)  $\beta = 0.75$ & (d)  $\beta = 0.90$
  \end{tabular}
  \caption{Comparing the unstable regions for NFLMM2 with other FLMMs for
  for $\beta = 0.25,0.50, 0.75, 0.90 $}\label{Fig_Comp_Stab}
\end{figure}
In Figure \ref{Fig_Comp_Stab}, the unstable regions for these FLMMs and our NFLMM2 are shaded for various values of $\beta $.
Note that the straight lines in the figures depicts the boundary of the stability region of the FT2 method in which
the left side of the lines are the stability regions
which are also correspond to the boundary of the analytical stability regions $\Sigma_\beta $. The unstable regions of FT2 arenot shaded for clarity.

The advantage of our NFLMM2 is, in terms of the unstable regions (UR), is that the UR of the NFLMM2 is smaller than that of FAM1 and is very much close to the UR of the FBDF2. Also, the UR of the FT2 is the largest among all the URs.

We note this from the observation that for the unstable regions ( see also the figures in Figure \ref{Fig_Comp_Stab} )
\[
\delta_{FBDF2}(-1) < \delta_{NFLMM2}(-1) < \delta_{FAM1}(-1) < \delta_{FT2}(-1)=+\infty .
\]

Another interesting observation is that, as $\beta $ approaches 1, the UR of FAM1 rapidly expands to the unbounded UR of FT2 while the UR of our NFLMM2 gets closer to the bounded UR of FBDF2 with very slow expnasion.

This is confirmed from the fact, as $\beta$ approaches 1,  that the generating function of the NFLMM2 converges to that of the FBDF2 while the generating function of the FAM1 converges to that of the FT2.

As for computational efficiency,
the weights $w_k$ of NFLMM2 has the simplest computational effort as they involve only a linear combinations the Gr\"unwald weights $g_k^{(\beta)} $ (see \eqref{Eq:FLMM_weights}).

Obviously, the weights of FBDF2 requires computations
using the Miller's formula (see eg. \cite{galeone2008fractional} ) with two previous weights.

The weights of FAM1 can be computed with the same amount of computation as that of NFLMM2. However, the right side of FAM1 scheme requires two coefficients from the Newton-Gregory expansion \cite{galeone2008fractional}.

Finally, the weights of FT2 need more efforts as they
require the first $n$ coefficients  of its generating function and requires FFT to compute \cite{garrappa2015trapezoidal}.

\normalsize

\section{Conclusion}\label{Sec_Conclusion}

We proposed and analysed a new FLMM of order two for FIVPs that falls under a new type of FLMM that is different from previously known types. The new FLMM is $A$-stable as the other known order two methods. However, the proposed method outweighs the other methods in terms of stability and/or computational cost.

\bibliographystyle{acm}

\bibliography{FLMM}

\end{document}